  \documentclass[amscd,amssymb,verbatim,12pt]{amsart}
\usepackage{epsfig}
\usepackage{graphicx}
\newif\ifAMS
\IfFileExists{amssymb.sty}
  {\AMStrue\usepackage{amssymb}}
  {\usepackage{latexsym}}
  \usepackage{enumerate}
\theoremstyle{plain}

\newtheorem{Thm}{Theorem}[section]

\newtheorem{Lem}[Thm]{Lemma}
\newtheorem{Prop}[Thm]{Proposition}

\theoremstyle{definition}
\newtheorem{Def}{Definition}

\newtheorem{Cl}{Claim}

\theoremstyle{remark}
\newtheorem{Rem}{Remark}

\newtheorem{Qu}[Thm]{Question}

\DeclareMathOperator{\area}{area}
\DeclareMathOperator{\length}{length}

\newcommand{\interior}{^{ \kern-5pt ^\circ}}
\newcommand {\bd}{\partial}

\begin{document}
\title
{A surface with discontinuous isoperimetric profile and expander manifolds}

\author
{Panos Papasoglu, Eric Swenson }


\address  [Panos Papasoglu]
{Mathematical Institute, University of Oxford, Andrew Wiles Building, Woodstock Rd, Oxford OX2 6GG, U.K.  }
\email {} \email [Panos Papasoglu]{papazoglou@maths.ox.ac.uk}

\address
[Eric Swenson] {Mathematics Department, Brigham Young University,
Provo UT 84602}
\email [Eric Swenson]{eric@math.byu.edu}


\begin{abstract} 
We construct sequences of `expander manifolds' and we use them
to show that there is a complete connected 2-dimensional Riemannian manifold with discontinuous isoperimetric profile,
answering a question of Nardulli and Pansu. Using expander manifolds in dimension 3 we show that for any $\epsilon , M>0$ there is a Riemannian 3-sphere $S$ of volume 1, such that any 
(not necessarily connected) surface
separating $S$ in two regions of volume greater than $\epsilon $, has area greater than $M$. 


\end{abstract}
\maketitle
\section{Introduction}

Expander graphs have been studied widely and they have important applications both in pure mathematics and in computer science. 
One may construct expander-like Riemannian manifolds with interesting properties (see e.g. \cite{Bu}, \cite{Br}, \cite{BI}). We will treat
two instances of such manifolds in this paper. We think it might be useful to give a formal definition of `expander manifolds':

\begin{Def} Let $M_n^k$ be a sequence of closed $k$-dimensional Riemannian manifolds all of which have volume 1. We say that
it is an expander sequence (or an expander manifold) if for any $n$ any smooth hypersurface $S_n^{k-1}$ separating $M_n^k$ in two open sets both
of which have volume at least $1/n$, has $(k-1)$-volume at least equal to $n$.
\end{Def}

Of course there are variations of this concept: e.g. one could allow manifolds with boundary or require that all $M_n^k$ are homeomorphic to a given manifold or share
a common geometric property.

If $(M^n,g)$ is a Riemannian manifold  the
isoperimetric profile function of $M^n$ is a function $I_M:(0,vol (M)) \to \Bbb R ^+$ defined by:
$$ I_M(t)=\underset {\Omega} { \inf } \{vol_{n-1}(\bd \Omega):
\Omega\subset M^n,\, vol_n(\Omega)=t \}$$ where $\Omega $ ranges
over all open sets of $M^n$ with smooth boundary. We note that we don't require $\Omega $ to be connected.

It is easy to see that $I_M(t)$ is upper semi-continuous by adding or taking away a small ball with
smooth boundary from a given region. More precisely to show upper continuity it suffices to
show that for any $t_0\in (0,vol (M))$ and $\delta >0$, there is an $\epsilon >0$ such that
$I_M(t)< I_M(t_0)+\delta $ if $|t-t_0|<\epsilon $. Given some open set $\Omega $ with $vol_n(\Omega)=t_0$ and
$vol_{n-1}(\bd \Omega)<I_M(t_0)+\dfrac{\delta}{2}$ we may pick $\epsilon $ small enough such that 
for any $\epsilon '<\epsilon $ there are 2 small open balls with smooth boundary $B_1,B_2$ with $B_1\subseteq \Omega $ and
$B_2$ contained in the complement of $\Omega $ with $vol_n(B_1)=vol_n(B_2)=\epsilon '$ and $vol_{n-1}(\bd B_1), vol_{n-1}(\bd B_2)<\delta/4$.
Then by considering the open sets $\Omega \setminus B_1$,  $\Omega \cup B_2$ we see that $I_M(t)< I_M(t_0)+\delta $ if $|t-t_0|=\epsilon '$.

The continuity of the isoperimetric profile function is established in a number of cases: Gallot \cite{Ga} showed that it is continuous
for compact manifolds and Nardulli-Russo \cite{NR} showed that this holds also for manifolds of finite volume. Ritor\'e \cite{Ri} showed
that $I_M$ is continuous for Hadamard manifolds and complete non-compact manifolds of strictly positive sectional curvature. We refer to
\cite{Ri} and \cite{NP} for a more complete exposition of the cases where the continuity of $I_M$ is established and related questions.

Adams, Morgan and Nardulli observed that $I_M(t)$ is not necessarily continuous for manifolds with density (see Frank Morgan's blog, \cite{AMN}).

Nardulli and Pansu \cite{NP} constructed an example of a Riemannian 3-manifold with discontinuous $I_M(t)$ and asked whether there is such an example in dimension 2.

It turns out that all these constructions are based on the existence of expander manifolds, so we answer this question in section 2 of this paper by constructing an appropriate
sequence of `expander surfaces' using expander graphs.  We note that P. Buser \cite{Bu} has used in the past expander graphs 
as `blueprints' for constructing surfaces. Psaltis \cite{Ps} used expander graphs to show that there are surfaces $M$ with no restrictions
on the growth rate of $I_M$.

The first author asked in \cite {Pa2} whether there are expander manifolds $M_n^3$ where each $M_n^3$ is homeomorphic to the 3-sphere (or to the 3-ball). It turns out that
such constructions were already given by Burago-Ivanov in \cite{BI}.

Glynn-Adey and Zhu show in \cite{GAZ} that for any $\epsilon >0, M>0$ there is a Riemannian 3-ball $B$ of volume 1
such that any smooth disk separating $B$ in two regions of volume greater than $\epsilon $ has area greater than $M$. We prove the
same result here both for the 3-ball and the 3-sphere for separations by arbitrary surfaces and not just disks. Glynn-Adey and Zhu assume further that the ball $B$ has bounded diameter and boundary surface area but these are properties that are easy to arrange in general modifying slightly the ball $B$. Our construction is different from
the one in \cite{BI}, for example our sequence of manifolds converges to a point.

These results contrast with the situation in dimension 2.
Liokumovich, Nabutovsky and Rotman showed in \cite{LNR} that if $D$ is a Riemannian 2-disc there is a simple
arc of length less than $2\sqrt 3 \sqrt {\rm{area}( D)}+\delta $ which cuts the disc into two regions of area greater than $\rm{area}( D)/4 -\delta $ where
$\delta $ is any positive real. A similar result was shown in \cite{Pa} for the sphere. The results in \cite{LNR} were prompted by a question
of Gromov \cite{Gr} and Frankel-Katz \cite{FK} concerning bounding the length of contracting homotopies of a 2-disk. 

Balacheff-Sabourau \cite{BS} showed that there is some $c>0$ such that any Riemannian surface $M$ of genus $g$ can be separated in two domains of equal area by a 1-cycle of length smaller than $c\sqrt {g+1}\sqrt {\rm{area}( M)}$. Liokumovich \cite{Li} on the other hand showed that given $C>0$ and a closed surface $M$ there is a Riemannian metric of diameter 1 on $M$ such that any 1-cycle splitting it into two regions of equal area has length greater than $C$. 

We thank the referee for a careful reading of the paper and several suggestions that improved the exposition.

\section {Surfaces with discontinuous profile}

We follow the same method as \cite{AMN}, \cite{NP}. Given some constants $a,b>0$  it is enough to find a sequence of closed surfaces $S_n$ with $\area (S_n)=a+\tau _n$ where $\tau _n\to 0$
such that $I_{S_n}(a)=I_{S_n}(\tau _n) > b$. Indeed we may join the surfaces $S_n$ by tubes of negligible area to obtain a surface $S$ such that
$I_S(a+\tau _n)\to 0$ but $I_S(a)> b$.

We explain now how to construct $S_n$.

Our construction relies on the existence of expander graphs. We recall the definition of expanders.
Let $\Gamma=(V,E)$ be a graph.
For $S,T\subseteq{V}$ denote the set of all edges between $S$ and $T$ by $$E(S,T)=\{(u,v):u\in{S},v\in{T},(u,v)\in{E}\}.$$
Recall that a graph is called $k$-regular if every vertex is contained in exactly $k$ distinct edges.

\begin{Def} The \textit{edge boundary} of a set $S\subseteq V$, denoted $\partial S$ is defined as $\partial S=E(S,S^c)$.\\
A $k$-regular graph $\Gamma=(V,E)$ is called a $c$-\textit{expander graph} if for all $S\subset V$ with $|S|\leq |V|/2$, $|\partial S|\geq c|S|$.
\end{Def}
Pinsker \cite{Pi} (see also \cite{SLW}, sec. 2.2) has shown that there is a $c>0$ such that for any $n$ large enough there is a 3-regular $c$-expander graph with $n$-vertices. 

In the course of the proof that follows we will need to show several inequalities; as obtaining best
constants is irrelevant for our construction we will always be using inequalities that are easy to state and verify, rather than optimal ones.
Although we could use any family of $k$-regular expander graphs it is somewhat easier to describe our examples using $3$-regular expander graphs.

Consider a 3-regular $c$-expander graph $\Gamma _n$ with $n^2+n$ vertices. We give a way to 
replace this graph by a Riemannian surface. 
 For each vertex $v$ we pick a Euclidean 3-sphere $S_v$ of radius $1/n$. Recall that the area of this
sphere is $4\pi (1/n)^2$. If $l$ is a great circle of $S_v$
we pick 3 equidistant points $e_1,e_2,e_3$ on $l$ and we consider 3 spherical caps $D_v^1,D_v^2,D_v^3$ on $S_v$ with centers $e_1,e_2,e_3$
and heights equal to $\dfrac {1}{10n}$. Clearly these spherical caps are disjoint and it is easy to see that the distance between any two of them on the sphere is greater than $1/n$.

Indeed if $D_v^i$ intersects $l$ at $x_i,x_i'$ and $\theta= \angle \, x_iOe_i$ (where $O$ is the center of $S_v$) then $\cos\theta=9/10>\cos (\pi/6)$
so the angle $  \angle \, x_iOx_j$ is greater than $\pi /3$ hence $d(x_i,x_j)>1/n$ on the sphere (for $i\ne j$).

Let $C_v^1,C_v^2,C_v^3$ be the boundary curves of $D_v^1,D_v^2,D_v^3$ respectively. We note that $$\length(C_v^i)\geq 1/n.$$

We remove the open spherical caps $D_v^1,D_v^2,D_v^3$ from $S_v$
and we still denote this sphere with holes by $S_v$ to keep notation simple.
We note that $$\area (S_v)=\dfrac{4\pi}{n^2}-3\dfrac{2\pi}{10n^2}=\dfrac{17\pi}{5n^2}.$$
We set $a=\dfrac{17\pi}{5}$, so $\area (S_v)=a/n^2$.

Now to each edge $E_i$ ($i=1,2,3$) in $\Gamma $ leaving $v$ we associate the boundary curve $C_v^i$. If an edge $e$ joins the vertices $v,w$ of $\Gamma $ we identify the
corresponding boundary curves of the spheres with holes $S_v,S_w$ by an isometry. This gluing produces a closed surface $S_n$ with
$$\area (S_n)=\dfrac{a(n^2+n)}{n^2}=a+\dfrac{a}{n}.$$
We set $\tau _n=\dfrac{a}{n}$.

We note that $S_n$ is not a smooth manifold but it is easy to slightly modify the metric to make it smooth and this modification with have a negligible effect on the 
following calculations, so we will pretend for the moment that $S_n$ is a smooth manifold.
We set $b=\min (\dfrac{1}{10},\dfrac{c}{4})$.
\begin{Prop}
Let $\Omega $ be an open set with smooth boundary in $S_n$ with $\area (\Omega)=\tau _n$. Then $$\length (\partial\Omega )\geq b.$$

\end{Prop}

\begin{proof}

We distinguish two types of spheres with holes $S_v$.

\textit{Type 1:}  $\length (\Omega \cap (C_v^1\cup C_v^2\cup C_v^3))\geq \dfrac{5}{2}\length (C_v^1)$ and

\textit{Type 2:} where the opposite inequality holds.

Let's denote by $\Omega _1$ the intersection of $\Omega $ with all type 1 spheres and $\Omega _2$ its intersection with type 2 spheres.

\begin{Lem}
Let $S_v$ be a type 2 sphere and let $\Omega _v=\Omega \cap S_v$ and $\delta \Omega _v=\partial \Omega\cap S_v$ . We have then

$$\area (\Omega _v)\leq \dfrac {8\pi}{n}\length (\delta \Omega _v)  .$$

\end{Lem}

\begin{proof} 

Note that $\area \Omega _v\leq \area S_v=a/n^2$. We note also that $\delta \Omega _v$ is in fact the boundary of $\Omega _v$ if we consider it to be
a subset of $S_v$ rather than $S$-for consistency's sake and to avoid ambiguity we will always consider our domains to be subsets of $S$.

We remark that our definition ensures that $\delta \Omega _v\subseteq \bd \Omega $ which
is convenient in what follows. 

Clearly it suffices to prove the inequality for all connected components $U$ of $\Omega _v$, namely it suffices to show:
$$\area (U)\leq \dfrac {8\pi}{n}\length (\delta \Omega _v\cap \bd U).\ \  \ \ \  \ \  (*) $$

We consider the connected components of $\delta \Omega _v$. If there is a connected component $\alpha $ that is an arc with endpoints in
two distinct $C_v^i$, $C_v^j$ then $\length (\alpha )\geq 1/n$ so the lemma holds as $$\dfrac{a}{n^2}\leq \dfrac{8\pi}{n}\dfrac{1}{n}.$$ 

Similarly if there is a connected component $\alpha $ that is a simple closed curve such that for some $i\ne j$ $C_v^i,C_v^j$ lie in the closures of 
distinct connected components of $S_v\setminus \alpha $ then $\length (\alpha )\geq 1/n$ and the lemma holds. So we may assume that $\delta \Omega _v$
contains no such simple closed curve $\alpha $.

Note also that if there is a connected component of $\delta \Omega _v$, $\alpha $, which is a simple closed curve that bounds a disc $D$ in $S_v$ which is not contained entirely in $\Omega _v$ then
we may delete $\alpha $. This will increase $\area (\Omega _v)$ and decrease $\length (\delta \Omega _v)$. So it suffices to prove the lemma for regions $\Omega _v$
such that $\delta \Omega _v$ contains no such simple closed curves.

Similarly suppose that there is a connected component $U$ of $\Omega _v$ such that $\bd U\cap \delta \Omega _v$ contains an arc $\alpha $ with its endpoints on some $C_v^i$ so that if $\beta $ is the arc with the same endpoints on $C_v^i$, $\alpha \cup \beta $ bounds a disc on $S_v$ which contains $U$. In this case if we erase all arcs in $\bd U\cap \delta \Omega _v$ except $\alpha $
the area of $U$ increases while the boundary length decreases (see figure 1). So we may assume for the purposes of showing $(*)$ that if a region $U$ intersects exactly one of the $C_v^i$ then
$\bd U\cap \delta \Omega _v$ is a single arc.


\begin{figure}[htbp]
\hspace*{-3.3cm}     
\begin{center}
                                                      
   \includegraphics[scale=0.500]{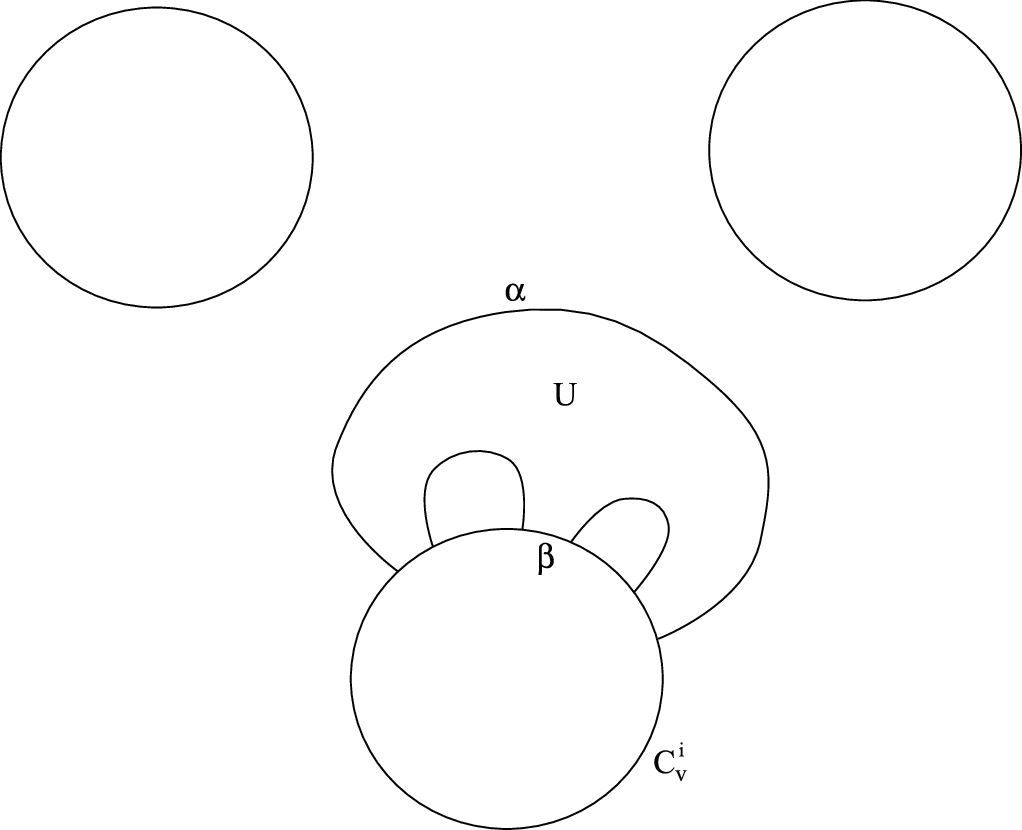}%
   \end{center}

\caption{Erasing the `inner' arcs.}
  \label{}
\end{figure}


To summarize if a connected component $U$ of $\Omega _v$ has a simple closed curve $\alpha $ in its boundary then we may assume that $U$ is a disc with boundary $\alpha $.
Also if some connected component $U$ of $\Omega _v$ intersects only one $C_v^i$ then we may assume (for the purpose of showing $(*))$ that $\bd U \cap \delta \Omega _v$ is a single arc. 
Finally if some connected component $U$ of $\Omega _v$ intersects two $C_v^i$'s then it intersects all $C_v^1, C_v^2, C_v^3$.

From these observations we conclude that it is enough to show $(*)$ assuming that the connected components $U$ of $\Omega _v$ fall in the following 3 cases (see figure 2):

\begin{figure}[htbp]
\hspace*{-3.3cm}     
\begin{center}
                                                      
   \includegraphics[scale=0.500]{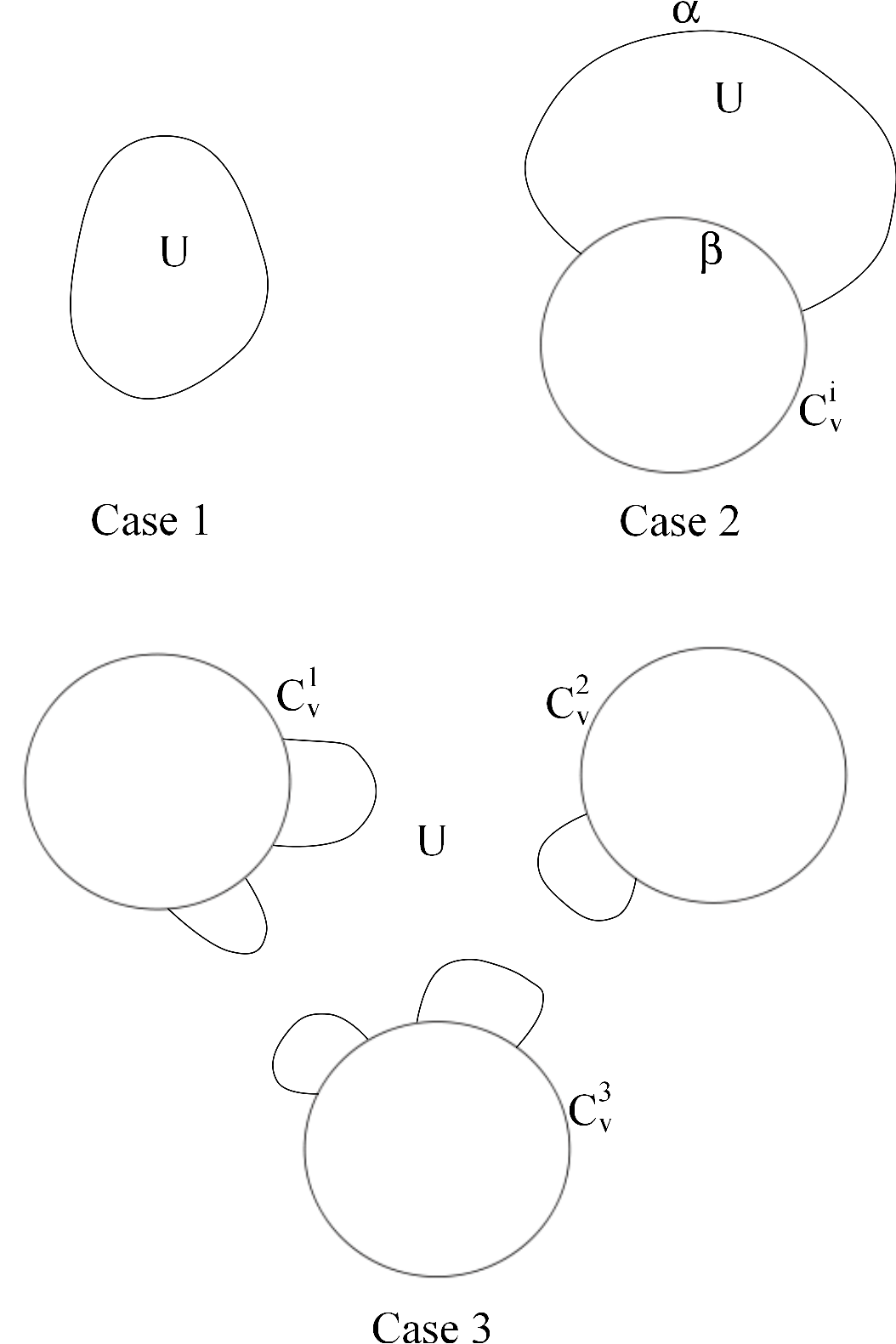}%
   \end{center}

\caption{}
  \label{}
\end{figure}

1. Regions $U$ homeomorphic to a disc with $\bd U\subseteq \delta \Omega _v$. By the isoperimetric inequality of the sphere the minimal boundary length for such regions
is for spherical caps. However the corresponding spherical cap has area less than the sphere with holes $S_v$ so it has
height $h$ less than $\dfrac {19}{10n}$. Recall that the area of a spherical cap is $2\pi r h$ where $r$ is the radius of the sphere and $h$ the height.  It follows that the length of the boundary of the spherical cap
is greater or equal to $h/2$ so

$$\area U\leq \dfrac{2\pi h}{n}\leq \dfrac{4\pi \length (\bd U)}{n}$$

 and the inequality holds.

2. Regions $U$ homeomorphic to a disc with $\bd U=\alpha \cup \beta$ where $\alpha ,\beta$ are arcs with the same endpoints and
$\alpha \subseteq \delta \Omega _v,\beta \subseteq C_v^i$ for some $i$. Then clearly $\length (\alpha )> \length (\beta )$
and applying the isoperimetric inequality of the sphere again we have

$$\area (U)\leq \dfrac{8\pi \length (\bd U)}{n} $$ 
and the inequality holds.

3. Regions $U$ homeomorphic to a sphere with 3 holes and $\bd U=\alpha \cup \beta$ where $\alpha $ is a union of arcs in $\delta \Omega _v$ with endpoints
on the $C_v^i$ 's and $\beta $ is a union of arcs contained in the $C_v^i$ 's. Since $S_v$ is of type 2 $\length (\alpha )\geq \dfrac{1}{2n} $
while $$\area (U)\leq \area (S_v)=\dfrac{a}{n^2}$$ and the inequality holds in this case too.


%


\end{proof}

We note that at least one of $\area (\Omega _1), \area (\Omega _2)$ is greater than or equal to $\tau _n/2$. We will show that in both cases the proposition holds.

Let's say that $\area (\Omega _2)\geq \tau _n/2$. If $a_v$ is the area of $S_v\cap \Omega _1$ and $\ell _v$ is the length of
$\bd \Omega _1\cap S_v$ then
$$\dfrac {\tau _n}{2}\leq \sum a_v\leq \dfrac{8\pi}{n} \sum \ell _v $$ so
$$\length (\partial\Omega )\geq \sum \ell _v > \dfrac{1}{10}\geq b .$$

Let's say now that $\area (\Omega _1)\geq \tau _n/2$. Let $T$ be the set of vertices of $\Gamma _n$ that correspond to spheres with holes of type 1.
Clearly $T$ contains at least $n/2$ vertices. 

\begin{Cl} $T$ contains less than $\dfrac {n^2+n}{2}$ vertices. 
\end{Cl}

\begin{proof}

Indeed if $T$ contains more than $\dfrac {n^2+n}{2}$
vertices then there are at least $n^2/10$ vertices $v$ of type 1 such that $\area (S_v\cap \Omega _1)<\area (S_v)/2$, otherwise we would have $\area (\Omega _1)> \tau _n$ which is impossible. We show now that for each sphere $S_v$ in $T$ for which $\area (S_v\cap \Omega _1)<\area (S_v)/2$
we have that $\length (\bd \Omega _1\cap S_v)\geq 1/n$. Indeed if $\bd \Omega _1\cap S_v$ contains an arc with endpoints in distinct $C_v^i$'s this is obvious. We note also that
none of the $C_v^i$'s is entirely contained in the complement of $\Omega _1$ as $S_v$ is of type 1. So $\bd \Omega _1\cap S_v$ consists of some arcs $\alpha _j$ $(j\in J)$ such that both endpoints
of the $\alpha _j$ lie in a single $C_v^i$. If we denote by $\beta _j$ the shorter of the two arcs on $C_v^i$ with the same endpoints as $\alpha _j$ then we note that $\alpha _j\cup \beta _j$ bounds a region of $S_v$ homeomorphic to a disc. If this is not
contained in some hemisphere then $\length (\alpha _j)\geq 1/n$ so $\length (\bd \Omega _1\cap S_v)\geq 1/n$. It follows by the
soperimetric inequality of the sphere that
$\length (\alpha _j)+\length (\beta _j)$ is greater than the boundary length of the spherical cap bounding the same area.
We note that
the union
$$\bigcup _{j\in J} (\alpha _j\cup \beta _j)$$ bounds a region of area greater than $\area (S_v)/2$. It follows  that 
$$\sum _{j\in J} (\length (\alpha _j)+\length (\beta _j))\geq 2\length (C_v^1).$$

Since $S_v$ is of type 1 $$\sum _{j\in J}\length (\beta_j)\leq \length (C_v^1)/2.$$
Hence $$\sum _{j\in J}\length (\alpha _j)\geq \length (C_v^1) \geq 1/n$$

So
$$\length (\bd (\Omega _1))\geq n/10.$$ This proves the claim, from which the proposition clearly follows.

\end{proof}

By our hypothesis that $\Gamma _n$ is a sequence of $c$-expander graphs there are at least $\dfrac {cn}{2}$ edges in the boundary
of $T$. 
If $(v,w)$ is such an edge, with $v\in T$, then for some $i\in \{1,2,3\}$
$C_v^i$ is identified to $C_w^j$ for some $j\in \{1,2,3\}$ and $S_v$ is of type 1 while $S_w$ is of type 2. Since $S_v$ is of type 1 
$$\length (C_v^i\cap \Omega)\geq \dfrac {1}{2}\length C_v^i$$  
so as $C_v^i$ is identified to $C_w^j$ $$\length (C_w^j\cap \Omega)\geq \length (C_w^j)/2.$$
On the other hand since $S_w$ is of type 2 $$\sum _{i=1}^3\length (C_w^i\setminus \Omega)\geq \length (C_w^j)/2.$$
If we set $$A=C_w^j\cap \Omega $$ and $B=(\bigcup _{i=1}^3C_w^i)\setminus \Omega $ we have that $\bd \Omega \cap S_w$ separates $A$ from $B$.

We claim now that $\bd \Omega \cap S_w$ has length at least $1/2n$.

If $\bd \Omega \cap S_w$ contains some simple closed curve $\alpha $ such that each region of $S_w\setminus \alpha $ contains some $C_w^i$ then clearly $\bd \Omega \cap S_w$ has length
more than $1/2n$. The same holds if $\bd \Omega \cap S_w$ contains some arc with endpoints in distinct $C_w^i$'s so in these cases the claim is proven.

Otherwise since $\bd \Omega $ separates $A,B$ at least one of the following two holds:

1. There is a set of arcs $\beta_j$ ($j\in J$) on $\bd \Omega $ with endpoints on $C_w^j$  such that there are arcs $\alpha _j$ on $C_v^j$ where $\alpha _j$ has the same endpoints as $\beta _j$,
$$A\subseteq \bigcup _{j\in J} \alpha _j$$ and for each $j\in J$ $\alpha _j\cup \beta _j$ bounds a disc in $S_w$.

2. There is a set of arcs $\beta_j$ ($j\in J$) on $\bd \Omega $ with endpoints on $\bigcup _{i=1}^3 C_w^i$  such that there are arcs $\alpha _j$ on $\bigcup _{i=1}^3 C_w^i$ where $\alpha _j$ has the same endpoints as $\beta _j$,
$$B\subseteq \bigcup _{j\in J} \alpha _j$$ and for each $j\in J$ $\alpha _j\cup \beta _j$ bounds a disc in $S_w$.

In both cases $$\sum _{j\in J} \length \beta _j\geq \sum _{j\in J} \length \alpha _j\geq \min (\length A,\length B)\geq 1/2n.$$
 It follows that in all cases
$\bd \Omega \cap S_w$ has length at least 
$1/2n$. Hence $$\length (\bd \Omega)\geq  \dfrac {c}{4}\geq b$$ from which the proposition follows.

\end{proof}

\section{A sphere hard to cut}

\begin{Def} Let $B$ be a Riemannian 3-ball. If $F\subset B$ is a smoothly embedded orientable surface with boundary
we say that $F$ \textit{separates} $B$ if $ F\cap \partial B=\partial F$. 

If $F$ is a surface separating a Riemannian 3-ball $B$ we say that $F$ \textit{cuts an $\epsilon $-piece} of $B$
if $B-F$ can be written as a union of two disjoint open sets $U,V$ both of which have volume greater than $\epsilon $.

We define similarly what it means for a closed surface to cut an $\epsilon $-piece of a Riemannian 3-sphere.

\end{Def}
%
%

Consider a 3-regular $c$-expander graph $\Gamma _n$ with $n^3$ vertices. We give a way to `thicken' this graph, i.e.
replace it by a Riemannian handlebody. For each vertex we pick a Euclidean 3-ball $B_v$ of radius $1/n$. Recall that the volume of this
ball is $4\pi/3\cdot (1/n)^3$. Let $S_v$ be the boundary sphere of $B_v$. If $l$ is an equator   $S_v$
we pick 3 equidistant points $e_1,e_2,e_3$ on $l$ and we consider 3 disjoint (spherical) discs on $S_v$ with centers $e_1,e_2,e_3$
and radii equal to $1/n$. Clearly these discs are disjoint. Now to each edge $E_i$ in $\Gamma $ leaving $v$ we associate the disc
with center $e_i$. If an edge $e$ joins the vertices $v,w$ of $\Gamma $ we identify the discs of the balls $B_v,B_w$ corresponding to this edge.

In this way we obtain a handlebody $\Sigma _n$. Note that $\partial \Sigma _n\cap B_v$ is a pair of pants. We will refer to $B_v$ later
on as a filled in pair of pants and we will call the discs with centers $e_1,e_2,e_3$ on $S_v$ the holes of this pair of pants.
We note that the area of $S_v$ is $4\pi (1/n)^2$ and the area of $S_v$ minus the 3 spherical discs is 
$$4\pi (1/n)^2-6\pi (1/n)^2 (1-\sin 0.5)=\pi(1/n)^2(6\sin 0.5-2).$$

By changing the metric of $\Sigma _n$ slightly we get a smooth handlebody, denoted still by $\Sigma _n$, of volume $4\pi/3$. Finally by gluing appropriately thickened
discs to this handlebody we obtain a ball $B_n$. We may assume that this gluing operation changes
the volume of $B_n$ and the area of its boundary by a negligible amount. 
 We may pick a simple curve $\gamma $ on $\partial B_n$ such that every point of $\partial B_n$ is at distance at most $1/n$ from
$\gamma $. By gluing a thickened disk of diameter $1/n$ and negligible volume to $\partial B_n$ along $\gamma $ we obtain a new ball of arbitrarily small diameter.
We still denote this 3-ball by $B_n$. In fact it follows also directly by the properties of expander graphs that the diameter of $B_n$
is bounded.

We double $B_n$ along its boundary
to obtain a 3-sphere. By changing the metric slightly along the doubling locus we may ensure that we obtain a smooth sphere $S_n$ of volume
$8\pi /3$.

%
%
%
%
\begin{Thm} \label{ball} Given $\epsilon, M>0$ there is some $n$ such that any surface that cuts an 
$\epsilon $-piece of $B_n$ (or $S_n$) has area greater than $M$.

\end{Thm}
\begin{proof}
We may (and will) assume that $\epsilon <1/100$.
Let $F$ be a (not necessarily connected) surface cutting an $\epsilon $-piece of $B_n$. So $B_n-F=U_1\cup U_2$ with $U_1,U_2$ open 
of volume greater than $\epsilon $.  We denote
by $Q_1,Q_2$ the closures of $U_1,U_2$ respectively. Without loss of generality we assume that $vol(U_2)\geq vol(U_1)$.

We note that $B_n$ contains a handlebody $\Sigma _n$ which is a union of filled in pairs of pants $B_v$-one for each vertex of the graph $\Gamma _n$.
Clearly $S_v\cap \partial \Sigma _n$ is a pair of pants with 3 holes.

Let $B_v$ be one such (filled in) pair of pants. Its volume is  $4\pi/3 n^3$.  By the solution of the isoperimetric problem for a ball
(\cite {Ro}) if a surface
cuts an $\epsilon 4\pi/ 3n^3$ piece of $B_v$ then its area is greater than $(4\pi\epsilon / 3n^3)^{2/3}\geq \epsilon/ n^2 $.

Let's say that for $n_1$ filled in pairs of pants $F$ cuts an $\epsilon / n^3$ piece and that for $n_2$ filled in pairs of pants
more than $\dfrac {4\pi(1-\epsilon)}{3n^3} $ of their volume is contained in $U_1$. Since $vol (U_1)\leq vol (U_2)$ 
$$n_2\leq 2\epsilon n^3\leq n^3/2$$

We distinguish two cases.

\textit{Case 1}. $n_1\geq \epsilon n^3/2$. Since the area of intersection of $F$ with each one of these $n_1$ filled in pairs of pants is
greater than $\epsilon/ n^2 $ the area of $F$ is greater than $\epsilon ^2 n/2$ which clearly tends to infinity as $n\to \infty $.

\textit{Case 2}. $n_1< \epsilon n^3/2$. Since $vol (U_1)>\epsilon $ we have that $n_2\geq \epsilon n^3/2$.
Let's denote this set of $n_2$-filled pairs of pants by $A$. Let $B_v$ be in $A$, and let $U_v=B_v\cap U_1$. Since 

$$vol (U_v)\geq \dfrac {4\pi(1-\epsilon)}{3n^3} $$ by the Euclidean isoperimetric inequality the boundary of $U_v$ has area at least
$$ \dfrac{4\pi (1-\epsilon)^{2/3}}{n^2}.$$

Since $\epsilon <1/100$ it follows that if the area of
$F\cap B_v$ is less than $\epsilon /2n^2$ then $U_1$ intersects non-trivially all
3 holes of the filled-in pair of pants. 


In fact since the area of a spherical cap is given
by $2\pi rh$ where $r$ is the radius and $h$ the height, the area of the intersection of $U_1$ with a hole is greater than
$$\dfrac {2\pi}{4n^2}>\dfrac {1}{n^2} \ \ \ \ \ \ \ \ \ \ \ \ \ \ \ \ \ \  (*).$$

Let's denote by $A_1$ the set of filled-in pair of pants in $A$ for which the area of intersection of $F\cap B_v$ is more than
$\epsilon /2n^2$ and let $A_2=A-A_1$. We set $k_1=|A_1|,\, k_2=|A_2|$ and note that $$k_1+k_2 = n_2 \geq \frac {\epsilon n^3} {2}.$$ If $k_1\geq \epsilon n^3/4$ then we see that
the area of $F$ is greater than $\epsilon ^2 n/8$ which clearly tends to infinity as $n\to \infty $. Otherwise $k_2\geq \epsilon n^3/4$.
By the expander property (and since $k_2\leq n^3/2$) the (not necessarily connected) union of filled in pairs of pants in $A_2$, $\Sigma $, has a boundary that consists of at least
$$c k_2 \geq \dfrac {c \epsilon n^3} {4}$$ holes. Let $B_v$ be a filled-in pair of pants adjacent to one of these holes. Clearly $B_v$ intersects $U_1$.
We claim that the area of $B_v\cap F$ is at least $\epsilon/ n^2 $. This is clear if $F$ cuts an $4\pi\epsilon / 3n^3$ piece of $B_v$
or if $B_v$ lies in $A_1$. If this is not the case then more than $(1-\epsilon)4\pi/n^3 $ of the volume of $B_v$ is contained in $U_2$. 
Let $O_v$ be the center of $B_v$.
Let's denote by $C_r$ the sphere with radius $r$ and center $O_v$. Let $l_r$ be the length of the intersection
of $F$ with $C_r$. If $l_r>1/10n$ for all $r$ with $1/n>r>9/10n$ then by the co-area formula (\cite {Fe}, 3.2.22) the area of $F\cap B_v$
is greater than $1/100n^2>\epsilon /n^2$. Otherwise we consider an $r_0\in (9/10n,1/n)$ for which $l_{r_0}$ is smaller than $1/10n$. We consider
the portion $F_1$ of $F$ between $C_{r_0}$ and the boundary of $B_v$ and we fill the holes of $F_1$ lying on $C_{r_0}$ by minimal area discs.
The total area of these disks is smaller than $\dfrac {\pi}{100n^2}$.
Let's call the surface obtained this way by $F_2$. Note that $F_2$ separates $U_1\cap B_v$ from $O_v$. Let $f$ be the radial projection
from $O_v$ to $C_1=S_v$. Clearly $f(F_2)$ contains $S_v\cap U_1$ and by inequality (*) the area of $S_v\cap U_1$ is greater than $\dfrac {1}{n^2}$ . Also $f$ is Lipschitz with Lipschitz constant less than $2$.
So the area of $f(F_2-F_1)$ is less than  $\dfrac {\pi}{50n^2}$. It follows that the area of $F_1$ is greater than 
$$\dfrac {1}{4n^2}$$ so the area of $F\cap B_v$ is greater than $\epsilon /n^2$ in this case too.

It follows as before that the area of $F$ is at least

$$\dfrac {c \epsilon n^3} {4}\cdot \dfrac{\epsilon}{ n^2}=\dfrac {c \epsilon ^2 n} {4} $$
which clearly tends to infinity as $n\to \infty $.

The result for the 3-sphere $S_n$ follows immediately from $B_n$ as $S_n$ is a union of two copies of $B_n$ and if
a surface cuts an $\epsilon $-piece of $S_n$ is cuts an $\epsilon /2$ piece in one of these two copies of $B_n$.
Finally clearly we may normalize the volume of $S_n$, $B_n$ to 1.

\end{proof} 

\begin{Rem} In \cite{GAZ} it is assumed additionally that the surface area and the diameter of the ball $B_n$ is bounded.
However both these properties are easy to arrange. As for the surface area one may excise a small ball from the 3-sphere $S_n$ in the
proof above and obtain a ball $B$ such that the area of $\partial B$ is arbitrarily small. By construction $B_n, S_n$ have diameter less than 1.
In fact given any ball (in any dimension $\geq 3$) one can easily decrease its diameter by surgery: one may cut out a thickened simple curve
and glue back in a ball with small diameter. This has no effect on the volume- or separation properties of the ball.
Even though we stated our result only for dimension 3 the same construction applies for spheres (balls) of any dimension $n\geq 3$.
\end{Rem}

\section{Discussion and further questions}

A general question is what geometric properties might prevent a sequence of $k$-manifolds $M_n^k$ from being an expander.
For example it is easy to see that if $M_n^k$ have Ricci curvature bounded from below then  $M_n^k$ can not be an expander sequence.
Glyn-Adey and Zhu showed (\cite {GAZ}, thm 1.6) that $M_n^3$ is not an expander sequence if the diameters and the homological filling function of the manifolds $M_n^3$ are bounded.
It is not clear however whether the bound on the diameter is necessary. We recall that the 1st homological filling function of a manifold is defined by
$HF_1(\ell )$ is defined to be the supremum of all the filling areas of cycles of length $\ell $. One may ask:

\begin{Qu} Let $M_n^3$ be a sequence of 3-manifolds homeomorphic to the 3-sphere such that for all $M_n^3$ $HF_1(\ell )\leq c\ell ^2$ for a fixed $c>0$.
Is it true that $M_n^3$ is not an expander sequence?
\end{Qu}

It is clear that a sequence of expander manifolds can not `converge' (in the Gromov-Hausdorff sense) to a manifold. Could one make precise how these
manifolds `collapse'? 
%

Even though a lower bound on Ricci curvature rules out expanders, interestingly the question whether such manifolds have continuous isoperimeric profile is still
open (see \cite{NP}, question 3 and \cite {FN}).

\end{document}
\bye